\documentclass[12pt]{article}
\usepackage[utf8]{inputenc}
\usepackage{amsmath}
\usepackage{amssymb,amscd}
\usepackage{amscd}
\usepackage{amsmath,latexsym,amssymb,amsthm}

\usepackage{txfonts}
\usepackage{tikz}
\usepackage{bbm}
\usepackage{hyperref}
\usepackage{cite}
\usepackage{standalone}
\usepackage{tikz-qtree}
\usepackage{lineno}
\theoremstyle{plain}
\newtheorem{thm}[subsection]{Theorem}
\newtheorem{lem}[subsection]{Lemma}
\newtheorem{prop}[subsection]{Proposition}
\newtheorem{cor}[subsection]{Corollary}

\newcommand{\card}[1]{\ensuremath{|#1|}}
\newcommand{\Zp}[1]{\mathbb{Z}/p^{#1}\mathbb{Z}}

\theoremstyle{definition}

\newtheorem{rem}[subsection]{Remark}
\newtheorem{defn}[subsection]{Definition}

\def\a{\alpha}

\def\b{\beta}

\def\Z{\mathbb{Z}}
\def\ua{\underbar{a}}
\def\ub{\underbar{b}}

\def\a{\alpha}

\def\b{\beta}

 \title{On Hamiltonian cycles of power graphs of abelian groups}
\author{Himadri Mukherjee\footnote{email: himadrim@bits-pilani.goa.ac.in}}

\begin{document}
\maketitle

\begin{abstract}
In this article we discuss the question of presence of Hamiltonian cycle in the undirected power graph of a group. In the process we develop weighted Hamiltonian cycle concept and prove
a few general results regarding the Hamiltonian question.
\end{abstract}

\section{Introduction} Graphs associated to algebraic structures is in literature for a long time. Most notable of it all 
would probably be the Cayley graph associated to a group. Association of a graph to a group dates back to Cayley (see \cite{cayley}), and more recently to Dehn
in \cite{dehn}. Dehn reintroduced the Cayley graph of a group and proposed the word problem of the mapping class group of a surface. In the theory of Cayley graphs it
is a long standing problem to solve, the conjecture of Lovász, that a vertex transitive graph has a Hamiltonian path and in particular a
connected Cayley graph is Hamiltonian. It is an elementary exercise to write down a Hamiltonian cycle when an abelian group acts transitively on a graph. More generally
in the framework of Cayley graphs it is easier to get some results on existence of a  Hamiltonian cycle when the group is an abelian group. 

But the same problem for a non-abelian group still remains open despite repeated attempts by eminent mathematicians. So it is an
interesting problem to ask whether the power graph associated to a group is Hamiltonian. The answer to that question is no, as it
is trivial to see from the definition of a power graph that the power graph of $\mathbb{Z}/p\mathbb{Z}$ has a cut vertex when $p$ is
a prime.  The definition of power
graph of a finite group occurs in the works of Kelarev et.al in the paper \cite{kelarev}, where the directed power graph of a finite group 
was studied. For a finite group $G$ the graph $\mathcal{G}=(V,E)$ is the graph on vertex set $V=G$ where $(g,h)$ is a directed edge from $g$ to 
$h$ if there is a natural number $k$ such that $h=g^k$.  In the survey \cite{survey} one can find a very detailed bibliography of the articles
related to power graph of a finite group. 
 
In the paper \cite{ivy} the question of Hamiltonicity of undirected power graph of the group of units in the ring $\mathbb{Z}/n\mathbb{Z}$ is Introduced. 
It was shown in the same paper that when the number $n$ is a product
of more than one Fermat primes \cite{burton} the group of units $U(n)=\mathbb{Zn}^{*}$ as mentioned above does not give a Hamiltonian graph.
But the general question remains open whether the group of units in $\mathbb{Z}/n\mathbb{Z}=U_n$ gives a Hamiltonian graph. 
In this article we give results that will be instrumental to solve the question raised in \cite{ivy}. We also develop graph theoretic tools that are useful to
tackle such graphs, namely given in clusters of complete graphs. Also as a question in itself the weighted Hamiltonian question is very interesting and useful.

%%%%%%%%%%%%%%%%%%%%%%%%%%%%%%%%%%%%%%%%%%%%%%%%%%%%%%%%%%%%%%%%%%%%%%%%%%%%%%%%%%%%%%%%%%%%%%%%%%%%%%%%%%%%%%%%%%%%%%%%%%%%%%%%%%%%%%%%%%%%%%%%%%%%%%%%%%%%%%%%%%%%%%%%%%%%%%%%%%%%%%%%%%%%%%%%%%%%%%%%%%%%%%%%%%%%%%%%%%%%%%%%%%%%%%%%%%%%%%%%%%%%%%%%%%%%%%%%%%%%%%%%%%%%%%%%%%%%%%%%%%%%%%%%%%%%%%%%%%%%%%%%%%%%%%%%%

%%%%%%%%%%%%%%%%%%%%%%%%%%%%%%%%%%%%%%%%%%%%%%%%%%%%%%%%%%%%%%%%%%%%%%%%%%%%%%%%%%%%%%%%%%%%%%%%%%%%%%%%%%%%%%%%%%%%%%%%%%%%%%%%%%%%%%%%%%%%%%%%%%%%%%%%%%%%%%%%%%%%%%%%%%%%%%%%%%%%%%%%%%%%%%%%%%%%%%%%%%%%%%%%%%%%%%%%%%%%%%%%%%%%%%%%%%%%%%%%%%%%%%%%%%%%%%%%%%%%%%%%%%%%%%%%%%%%%%%%%%%%%%%%%%%%%%%%%%%%%%%%%%%%%%%%%

%%%%%%%%%%%%%%%%%%%%%%%%%%%%%%%%%%%%%%%%%%%%%%%%%%%%%%%%%%%%%%%%%%%%%%%%%%%%%%%%%%%%%%%%%%%%%%%%%%%%%%%%%%%%%%%%%%%%%%%%%%%%%%%%%%%%%%%%%%%%%%%%%%%%%%%%%%%%%%%%%%%%%%%%%%%%%%%%%%%%%%%%%%%%%%%%%%%%%%%%%%%%%%%%%%%%%%%%%%%%%%%%%%%%%%%%%%%%%%%%%%%%%%%%%%%%%%%%%%%%%%%%%%%%%%%%%%%%%%%%%%%%%%%%%%%%%%%%%%%%%%%%%%%%%%%%%

%%%%%%%%%%%%%%%%%%%%%%%%%%%%%%%%%%%%%%%%%%%%%%%%%%%%%%%%%%%%%%%%%%%%%%%%%%%%%%%%%%%%%%%%%%%%%%%%%%%%%%%%%%%%%%%%%%%%%%%%%%%%%%%%%%%%%%%%%%%%%%%%%%%%%%%%%%%%%%%%%%%%%%%%%%%%%%%%%%%%%%%%%%%%%%%%%%%%%%%%%%%%%%%%%%%%%%%%%%%%%%%%%%%%%%%%%%%%%%%%%%%%%%%%%%%%%%%%%%%%%%%%%%%%%%%%%%%%%%%%%%%%%%%%%%%%%%%%%%%%%%%%%%%%%%%%%
\section{Structure of Power Graphs}

In this section we define a few basic concepts and prove a few basic results for general power graphs of groups. 
\begin{defn}
 Power graph $P_G=(V,E)$ of a finite group $G$ is given by the vertex set $V$ and edge set $E=\{(g,h) : \exists k \in \mathbb{N} , g^k =h\}$.
\end{defn}
 Note that we have a natural direction on the edges of the power graph of a group. In this article we will consider only the undirected power 
 graph of a group as it is easy to observe that there cannot be any directed Hamiltonian cycle in a power graph (identity vertex is a sink). 
 But many a time we will consider the direction, although we will mention it whenever we intend it, to facilitate our study.

 \begin{lem}\label{edge_composition}
  In a directed power graph $G$ if $(x,y), (y,z)$ are two directed edges, then $(x,z)$ is a directed edge. 
 \end{lem}
\begin{proof}
 There are $n_x, n_y \in \mathbb{N}$ such that $x^{n_x}=y \mbox{ and } y^ {n _y }= z$ so putting these two together we have $x^{n_x n_y}=z$.
\end{proof}
Edges obtained in a process described in the above lemma will be called transitive edges.
\begin{lem} Let $G$ be a group, if $(g,h)$ and $(h,g)$, directed edges, exist in the power graph of $G$ then $o(g)=o(h)$. 
\end{lem}
\begin{proof}
 Note that order $o(h^k)$ of the element $h^k$ is $\frac{o(h)}{gcd(o(h),k))}$ and then it follows easily.
\end{proof}

\subsection{Cluster graphs}
Throughout this section we will consider the directed power graph of a group.
\begin{defn}A cluster in a power graph $G$ is a strongly connected component in the associated directed power graph.
 
\end{defn}
Note that a natural graph structure precipitates on the set of clusters $\mathcal{C}$ of the power graph  $G=(V,E)$ as the quotient 
graph under the equivalence relation defined by the strong 
components, namely if $C_1,C_2$ are two clusters $(C_1,C_2)$ is an edge if there are $x \in C_1, y \in C_2$ such that $(x,y) \in E$. 
Note that once such an edge exist there will be
edges between every pair $(z,w)$ , $z \in C_1, w \in C_2$.
Also note that this graph structure acquires a natural direction on the edges and further note that by the lemma \ref{edge_composition} this 
direction on the edges of the clusters is not reversible, otherwise the two clusters will be one unified strong componnent. 
We will call this the cluster graph of the power graph $G$.
The induced subgraph on the set of vertices of a cluster $c$ is a directed complete subgraph of the graph $G$, as a result the graph
$G$ is totally determined 
if the cluster graph and the size of the clusters are known. 
The cluster graph with the vertex weight defined by the size of the cluster totally characterizes the graph $G$. Let $G=(V,E)$ be a vertex
weighted graph with vertex weight $w: V \rightarrow \mathbb{N}$, 
we define $\widetilde{G^{w}}$ as a the graph on the set of vertices $\cup_{v \in V}\{v\}\times [w(v)]$ ($[n]=\{1,2, \ldots , n\})$ with $((v,i),(u,j))$ an edge 
if either $v=u$ and $i \neq j$ or $(v,u) \in E$, this operation is reversible since the original graph $G$ is the quotient graph of the graph $\widetilde{G^{w}}$
where $(v,i)$ is equivalent to
$(u,j)$ if $v=u$. Thus a power graph $G$ is isomorphic to the graph $\widetilde{\mathcal{C}^{w}}$ where $\mathcal{C}$ is the cluster graph and $w(c)=\card{c}$. 
We sum up the above discussion in the following three brief lemmas, the proofs of which are more or less automatic.
\begin{lem}
 If $c_1,c_2,c_3$ such that $(c_1,c_2)$ and $(c_2,c_3)$ are directed cluster edges, then there is a directed cluster edge $(c_1,c_3)$.
\end{lem}

\begin{lem}Let $c_1,c_2$ be two clusters in a graph $G$ and let $v_i \in c_i$ be vertices such that $(v_1,v_2)$ is an edge in $G$, 
then for every $v \in c_1$ and $v' \in c_2$ there are edges
$(v,v')$ in $G$.
\end{lem}
\begin{proof} Follows from the lemma \ref{edge_composition}.
 
\end{proof}
\begin{lem} Let $G$ be a directed graph, let $C(G)$ be its cluster graph with $w(c)=\card{c}$ vertex weights, then $G\simeq \widetilde{C(G)^w}$.
 
\end{lem}
 Throughout in this article, unless otherwise stated, we will assume that the product of two graphs is given in the following manner. 
 \begin{defn}\label{graph_product}
  Let $G_i=(V_i,E_i) \mbox{ for } i = 1,2 $ be two graphs, then the product graph is defined as the graph $G_1 \times G_2= ( V_1 \times V_2 , E)$ where 
  $((v_1,w_1), (v_2,w_2)$ defines an edge if 
  $v_1=v_2 \mbox{ and } (w_1,w_2) \in E_2 \mbox{ or } (v_1,v_2) \in E_1 \mbox{ and } w_1=w_2 \mbox{ or } (v_1,v_2) \in E_1 \mbox{ and } (w_1,w_2) \in E_2$. 
 \end{defn}

\begin{lem}
 Let $G_1,G_2$ be two directed graphs and $C(G_i)$ be the corresponding cluster graphs then $C(G_1\times G_2)=C(G_1) \times C(G_2)$.
\end{lem}

If a cluster edge is not obtained as a transitive edge then it is called an irreducible edge. 
The directed cluster graph with irreducible edges and the size of each cluster totally determines
the original graph since all the edges could be reconstructed by taking the transitive closure of 
the irreducible edges so the cluster graph with cluster sizes and irreducible edges totally determine 
the original graph. As a convention for drawing figures, whenever there is not a chance of ambiguity, our edges in a cluster graph will be oriented from bottom to top.
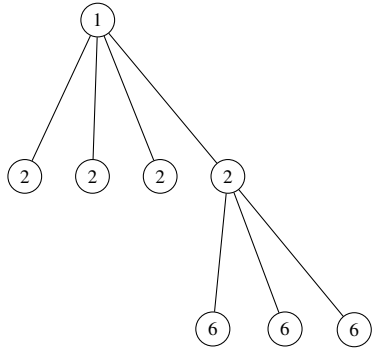
\begin{figure}
\centering
\begin{tikzpicture}[every node/.style={circle, draw, scale=.6}, scale=1.0, rotate = 180, xscale = -1]

\node (1) at ( 3.12, 2.77) {1};
\node (2) at ( 2.15, 4.85) {2};
\node (3) at ( 3.05, 4.85) {2};
\node (4) at ( 3.95, 4.85) {2};
\node (5) at ( 4.85, 4.85) {2};
\node (6) at ( 4.65, 6.88) {6};
\node (7) at ( 5.61, 6.88) {6};
\node (8) at ( 6.53, 6.89) {6};

\draw (2) -- (1);
\draw (3) -- (1);
\draw (4) -- (1);
\draw (5) -- (1);
\draw (6) -- (5);
\draw (7) -- (5);
\draw (8) -- (5);

\end{tikzpicture}
\caption{cluster graph with irreducible edges of $\mathbb{Z}_{3^2}\times \mathbb{Z}_3$}
\label{fig:1}
\end{figure}

 \begin{defn}
  A graph $G=(V,E)$ along with a map $f: E \rightarrow \mathcal{P}(\mathbb{N})$, is called a power graph if there is a group $H$ such that $G=P_H$ and $f(g,h)=\{k \in \mathbb{N} : g^k=h\}$.
 \end{defn}

 Let us also define a product of power graphs in the following manner which we will call strong product. 
 \begin{defn}\label{strong_prod}
  Let $G_1,G_2$ be two power graphs and let $f_1,f_2$ be two edge functions respectively, then we define a product
  $G_1 \boxtimes G_2=(G_1\times G_2, E_1 \boxtimes E_2)$ where $E_1 \boxtimes E_2 $ is given by the edges $((v_1,x_1),(v_2,x_2))$ is an 
  edge if $f_1(v_1, x_1) \cap f_2(v_2,x_2) \neq \emptyset$.
 \end{defn}

 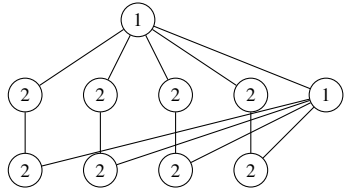
\begin{figure}[h]
\centering
\begin{tikzpicture}[every node/.style={circle, draw, scale=.6}, scale=1.0, rotate = 180, xscale = -1]

\node (21) at ( 0, 0) {2};
\node (22) at ( 1, 0) {2};
\node (23) at ( 2, 0) {2};
\node (24) at ( 3, 0) {2};
\node (25) at ( 0, -1) {2};
\node (26) at ( 1, -1) {2};
\node (27) at ( 2, -1) {2};
\node (28) at ( 3, -1) {2};
\node (29) at ( 4, -1) {1};
\node (30) at ( 1.5, -2) {1};

%\draw (30) --(21);
%\draw (30) -- (22);
%\draw (30) -- (23);
%\draw (30) -- (24);
\draw  (30) -- (25);
\draw (30) -- (26);
\draw (30) --(27);
\draw (30) -- (28);
\draw (30) -- (29);
\draw (21) -- (25);
\draw (22) -- (26);
\draw (23) -- (27);
\draw (24) -- (28);

\draw (21) -- (29);
\draw (22) -- (29);
\draw (23) -- (29);
\draw (24) -- (29);

\end{tikzpicture}
\caption{Cluster graph of $\mathbb{Z}_3 \times \mathbb{Z}_3 \times \mathbb{Z}_2$}
\label{fig:2}
\end{figure}

 \begin{thm}
  If the groups $G,G_1,G_2$ related as $G=G_1\times G_2 \mbox{ then } P_{G_1} \boxtimes P_{G_2}=P_{G} $, where $P_G$ denotes the power graph of the group $G$.
 \end{thm}
\begin{proof}
 Only thing that we have to prove is that the edge sets are equal for the both sides. In other words we have to prove that $E_1 \boxtimes E_2= E$, 
 where $E$ is that edge set of the graph $G$. Let $(x,y), (g,h) \in G \mbox{ such that there is an edge }$ $ ((x,y),(g,h)) \in E$. Or equivalently we
 have a natural number $k$ such that 
 $(x,y)^k=(g,h) \mbox{ or } x^k=g \mbox{ and } y^k =h \iff (x,g) \in E_1 \mbox{ and } (y,h) \in E_2 \mbox{ and } f_1(x,g) \cap f_2 (y,h) \neq \emptyset ( \mbox{ since it contains } k ) \iff ((x,y),(g,h)) \in E_1 \boxtimes E_2$. 
\end{proof}

 \begin{thm}
  Let $G_1, G_2$ be two groups with coprime orders then $P_{G_1} \boxtimes P_{G_2}= P_{G_1} \times P_{G_2}$, where $P_{G_1} \times P_{G_2}$ is the product of graphs defined in \ref{graph_product}.
 \end{thm}
 \begin{proof}
  We just have to show that $E_1 \boxtimes E_2 = E_1 \times E_2$ to show that it is enough to show that for any 
  $(x,y) \in E_1 \mbox{ and } (g,h) \in E_2 \, , f_1(x,y) \cap f_2(g,h) \neq \emptyset$. Equivalently, there exists $k_1$ and $k_2$ such that 
$x^{k_1}=y$ and $g^{k_2}=h$. So $f_1(x,y)=\{k_1+m_1 o(x) : m_1\in \mathbb{Z}\}$, $f_2(g,h)=\{k_2+m_2 o(g) :m_2\in \mathbb{Z}\}$, where $o(x)$ and $o(g)$
are prime to each other. So we have to show that there is a solution of the following equations: $x\equiv k_1(\mbox{ mod }o(x))$ and $x\equiv k_2(\mbox{ mod }o(g))$.
Now $o(x)$ and $o(g)$ being co-prime, by Chinese Remainder theorem we get a solution of the above equations say the solution be $k$. Therefore $x^k=y$ and $g^k=h \Rightarrow (x,g)^k=(y,h)$. Hence follows the theorem. 
 \end{proof}
The following few are a few relevant general theorems about Hamiltonian graphs that will be useful later. 
 \begin{thm} \label{product_hamiltonian}
  Let $G_1, G_2$ be two Hamiltonian graphs then $G_1 \times G_2$ is Hamiltonian. 
 \end{thm}
\begin{proof}
 
 Fix a Hamiltonian cycle $v_0,v_1,\ldots, v_n=v_0$ in the graph $G_1$, let us also fix a Hamiltonian cycle $w_0,w_1,\ldots, w_m=w_0$ in the graph $G_2$. 
 We define a Hamiltonian cycle in the graph $G_1 \times G_2$ in the following way. Define paths $\pi_i=(v_0,w_i),(v_1,w_i),\\ \ldots , (v_{n-1},w_i)$ on the product graph. 
 Note that there is an edge in the product graph between the end vertex of the path $\pi_i$ and the start vertex of the path $\pi_{i+1}$ for each $i \leq m-1$.
 And the end vertex of the path $\pi_m$ has an edge to the start vertex of the path $\pi_1$. So the concatenation $\pi_0\pi_1\ldots \pi_m$ is a Hamiltonian cycle in the product graph.
 
 \end{proof}

 \begin{defn}\label{equiv_graph}
  Let $G=(V,E)$ be a graph and let $\sim$ be an equivalence relation on the set of vertices $V$  such that if $v \sim w$ and $(w,w') $ is an edge then $v \sim w'$, 
  the quotient graph $G/\sim = (V', E')$ where $V'=\{[v] : v \in V \}$ 
  where $[v]$ stands for the equivalence class of the vertex $v$. $([v],[w]) \in E'$ if for any vertex $v_1 \in [v] \mbox{ there is a vertex } w_1 \in [w]$ such that $(v_1,w_1) \in E$. 
 \end{defn}

\begin{thm}\label{equiv_hamiltonian}
 Let $G=(V,E)$ be a graph and let $\sim$ be an equivalence relation on the set of vertices $V$.
 Let $[v]$ stand for the equivalence class of the vertex $v$.
 If the quotient graph $G/ \sim$ and the induced subgraphs on equivalence classes $[v]$ are Hamiltonian for all $v \in V$ then $G$ is Hamiltonian.
\end{thm}
\begin{proof}
 Let us fix a Hamiltonian cycle on the quotient graph $[v_1],[v_2],\ldots , [v_l]$. Now also fix a Hamiltonian cycle for each of the
 induced subgraphs $[v_i]$ let these be $w_{i1},w_{i2}, \ldots w_{ir_i}$, let us call this path $\pi_i$. So we have a cycle $\pi_1\pi_2, \ldots \pi_l$, 
 this is a path since from the end vertex of $\pi_i$ that is $w_{ir_i}$ there is an edge to the vertices $\{w_{(i+1)1},w_{(i+1)2},\ldots, w_{(i+1)r_{i+1}}\}$ 
 without loss of generality let us say to the vertex $w_{(i+1)1}$. So the concatenation works. From the last vertex $w_{lr_l}$ there is an edge to a vertex in 
 the set $\{w_{11},w_{12}, \ldots w_{1r_1} \}$ if it is not the vertex $w_{11} $ we can relabel so that we have a cycle. 
\end{proof}

%%%%%%%%%%%%%%%%%%%%%%%%%%%%%%%%%%%%%%%%%%%%%%%%%%%%%%%%%%%%%%%%%%%%%%%%%%%%%%%%%%%%%%%%%%%%%%%%%%%%%%%%%%%%%%%%%%%%%%%%%%%%%%%%%%%%%%%%%%%%%%%%%%%%%%%%%%%%%%%%%%%%%%%%%%%%%%%%%%%%%%%%%%%%%%%%%%%%%%%%%%%%%%%%%%%%%%%%%%%%%%%%%%%%%%%%%%%%%%%%%%%%%%%%%%%%%%%%%%%%%%%%%%%%%%%%%%%%%%%%%%%%%%%%%%%%%%%%%%%%%%%%%%%%%%%%%

%%%%%%%%%%%%%%%%%%%%%%%%%%%%%%%%%%%%%%%%%%%%%%%%%%%%%%%%%%%%%%%%%%%%%%%%%%%%%%%%%%%%%%%%%%%%%%%%%%%%%%%%%%%%%%%%%%%%%%%%%%%%%%%%%%%%%%%%%%%%%%%%%%%%%%%%%%%%%%%%%%%%%%%%%%%%%%%%%%%%%%%%%%%%%%%%%%%%%%%%%%%%%%%%%%%%%%%%%%%%%%%%%%%%%%%%%%%%%%%%%%%%%%%%%%%%%%%%%%%%%%%%%%%%%%%%%%%%%%%%%%%%%%%%%%%%%%%%%%%%%%%%%%%%%%%%%

%%%%%%%%%%%%%%%%%%%%%%%%%%%%%%%%%%%%%%%%%%%%%%%%%%%%%%%%%%%%%%%%%%%%%%%%%%%%%%%%%%%%%%%%%%%%%%%%%%%%%%%%%%%%%%%%%%%%%%%%%%%%%%%%%%%%%%%%%%%%%%%%%%%%%%%%%%%%%%%%%%%%%%%%%%%%%%%%%%%%%%%%%%%%%%%%%%%%%%%%%%%%%%%%%%%%%%%%%%%%%%%%%%%%%%%%%%%%%%%%%%%%%%%%%%%%%%%%%%%%%%%%%%%%%%%%%%%%%%%%%%%%%%%%%%%%%%%%%%%%%%%%%%%%%%%%%

%%%%%%%%%%%%%%%%%%%%%%%%%%%%%%%%%%%%%%%%%%%%%%%%%%%%%%%%%%%%%%%%%%%%%%%%%%%%%%%%%%%%%%%%%%%%%%%%%%%%%%%%%%%%%%%%%%%%%%%%%%%%%%%%%%%%%%%%%%%%%%%%%%%%%%%%%%%%%%%%%%%%%%%%%%%%%%%%%%%%%%%%%%%%%%%%%%%%%%%%%%%%%%%%%%%%%%%%%%%%%%%%%%%%%%%%%%%%%%%%%%%%%%%%%%%%%%%%%%%%%%%%%%%%%%%%%%%%%%%%%%%%%%%%%%%%%%%%%%%%%%%%%%%%%%%%%

\subsection{$p$ groups} From the previous subsection it is clear that to understand the power graph of a nilpotent group it is required to understand the primary 
components of it and the study of strong product of these graphs. One could at this point branch out to study the power graphs of general nilpotent groups but in 
this article we will only consider abelian groups as our main motive dictates that and some structural understanding of the primary abelian groups are easier. So in 
this subsection we will consider an Abelian $p$-group $G_p$ for a prime $p$, which using the structure theorem of Abelian group can be written 
as $\bigoplus_{i=1}^{n} \mathbb{Z}/p^{\a_i}\mathbb{Z}$ for some natural numbers $\a_i$. For $a_i \in \mathbb{N}_0 \mbox{ such that } a_i \leq \a_i$ we define 
$N_{(a_1,a_2, \ldots , a_n)}$ 
or simply by $N_{\overline{a}}$ to be the set $\prod N_{a_i}$ where $N_{a_i}=\{g \in \mathbb{Z}/p^{\a_i} \mathbb{Z} : o(g) = p^{\a_i -a_i} \}$. For a tuple 
$\overline{a}=(a_1,a_2,\ldots, a_n)$ we can define
$\overline{a}+1=(\mbox{min}\{a_1+1,\a_1\},\mbox{min}\{a_2+1,\a_2\}\dots,\mbox{min}\{a_{n}+1,\a_n\})$. Let us define a notation 
$I_{\a_1,\a_2, \ldots , \a_n}=\{ (\overline{a}=(a_1,a_2,\ldots , a_n) : \mbox{ where } 0\leq a_i \leq \a_i\}$, if we declare an edge between $\underline{a}$ 
and $\underline{a}+1$ then the
resulting graph structure on the set $I_{\underline{\a}}$ is a tree. For the Abelian $p$-primary group $G_p$ the tree $I_{\underline{\a}}$ is called 
$N_{\underline{\a}}$ tree of 
the group $G_p$.
\begin{rem}\label{cluster}
 \begin{enumerate} 
  \item There is a map $f: N_{\overline{a}} \rightarrow N_{\overline{a}+1}$ where $f(g_1,g_2,\ldots,g_n)=(pg_1,pg_2,\ldots, pg_n)$
  \item The group $H=(\mathbb{Z}/p^{a}\mathbb{Z})^{*}$ acts on the set $N_{\overline{a}}$ where $a=\mbox{max}\{a_1,a_2,\ldots,a_n \}$, by the action:
  $x \in H $ then $x*(g_1,g_2,\ldots, g_n)= (\overline{x}g_1, \overline{x}g_2, \ldots, \overline{x}g_n)$. Where $\overline{x}$ stands for the images of $x$ 
  under the maps $\mathbb{Z}/p^{a}\mathbb{Z} \rightarrow \mathbb{Z}/p^{a_i}$.
 \end{enumerate}

\end{rem}

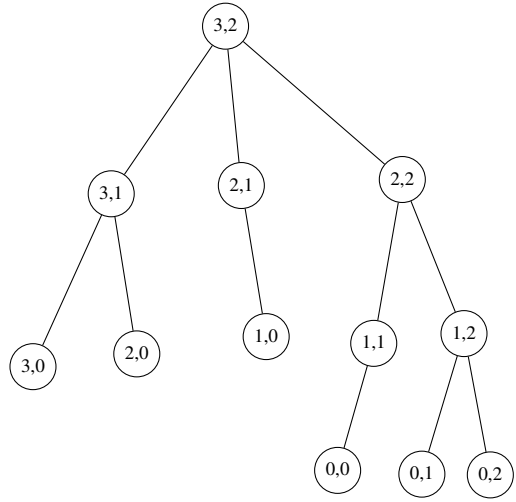
\begin{figure}
\centering
\begin{tikzpicture}[every node/.style={circle, draw, scale=.6}, scale=1.0, rotate = 180, xscale = -1]

\node (9) at ( 3.32, 2.97) {3,2};
\node (10) at ( 1.8, 5.2) {3,1};
\node (11) at ( 3.53, 5.09) {2,1};
\node (12) at ( 5.67, 5.01) {2,2};
\node (13) at ( 0.76, 7.5) {3,0};
\node (14) at ( 2.14, 7.33) {2,0};
\node (15) at ( 3.86, 7.1) {1,0};
\node (16) at ( 5.29, 7.19) {1,1};
\node (17) at ( 6.49, 7.06) {1,2};
\node (18) at ( 4.81, 8.88) {0,0};
\node (19) at ( 5.93, 8.93) {0,1};
\node (20) at ( 6.84, 8.94) {0,2};

\draw (10) -- (9);
\draw (13) -- (10);
\draw (14) -- (10);
\draw (11) -- (9);
\draw (15) -- (11);
\draw (12) -- (9);
\draw (16) -- (12);
\draw (17) -- (12);
\draw (18) -- (16);
\draw (19) -- (17);
\draw (20) -- (17);

\end{tikzpicture}
\caption{$N_{r,s}$ tree of $\mathbb{Z}_{p^3} \times \mathbb{Z}_{p^2}$}
\label{fig:3}
\end{figure}

\begin{lem}\label{cluster_def} under the action of the group $H=(\mathbb{Z}/p^a\mathbb{Z})^{*}$ on the $N_{\underline{a}}$ the orbits are precisely 
the clusters of the power graph of $G_p$.
 
\end{lem}
\begin{proof}
 Let $g=(g_1,g_2,\ldots, g_n) \in N_{\underline{a}}$ and let $O_g$ be the orbit through $g$. Let $x*g \in O_g$ for some $x \in H$, now choosing 
 an appropriate representative $x'$ in $\mathbb{Z}$
 for the element $x$, we write $x*g = (\overline{xg_1},\overline{xg_2},\ldots, \overline{xg_n})=(x'g_1,x'g_2,\ldots,x'g_n)=xg$. Hence there is an 
 edge between $g$ and $x*g$. Now note that the
 edge is reversible since for $y \in H$ such that $xy=1$ we have $y(x*g)=yx*g=g$ thus we have an edge from $x*g $ to $g$. Now if there is $z \in G$ 
 such that there is a reversible edge between 
 $z$ and $g$, then $z= kg$ for some $k \in \mathbb{N}_0$ hence $z$ is in the orbit of $g$.
\end{proof}

%%%%%%%%%%%%%%%%%%%%%%%%%%%%%%%%%%%%%%%%%%%%%%%%%%%%%%%%%%%%%%%%%%%%%%%%%%%%%%%%%%%%%%%%%%%%%%%%%%%%%%%%%%%%%%%%%%%%%%%%%%%%%%%%%%%%%%%%%%%%%%%%%%%%%%%%%%%%%%%%%%%%%%%%%%%%%%%%%%%%%%%%%%%%%%%%%%%%%%%%%%%%%%%%%%%%%%%%%%%%%%%%%%%%%%%%%%%%%%%%%%%%%%%%%%%%%%%%%%%%%%%%%%%%%%%%%%%%%%%%%%%%%%%%%%%%%%%%%%%%%%%%%%%%%%%%%%%%%%%%%%%%%%%%%%%%%%%%%%%%%%%%%%%%%%%%%%%%%%%%%%%%%%%%%%%%%%%%%%%%%%%%%%%%%%%%%%%%%%%%%%%%%%%%%%%%%%%%%%%%%%%%%%%%%%%%%%%%%%%%%%%%%%%%%%%%%%%%%%%%%%%%%%%%%%%%%%%%%%%%%%%%%%%%%%%%%%%%%%%%%%%%%%%%%%%%%%%%%%%%%%%%%%%%%%%%%%%%%%%%%%%%%%%%%%%%%%%%%%%%%

%%%%%%%%%%%%%%%%%%%%%%%%%%%%%%%%%%%%%%%%%%%%%%%%%%%%%%%%%%%%%%%%%%%%%%%%%%%%%%%%%%%%%%%%%%%%%%%%%%%%%%%%%%%%%%%%%%%%%%%%%%%%%%%%%%%%%%%%%%%%%%%%%%%%%%%%%%%%%%%%%%%%%%%%%%%%%%%%%%%%%%%%%%%%%%%%%%%%%%%%%%%%%%%%%%%%%%%%%%%%%%%%%%%%%%%%%%%%%%%%%%%%%%%%%%%%%%%%%%%%%%%%%%%%%%%%%%%%%%%%%%%%%%%%%%%%%%%%%%%%%%%%%%%%%%%%%%%%%%%%%%%%%%%%%%%%%%%%%%%%%%%%%%%%%%%%%%%%%%%%%%%%%%%%%%%%%%%%%%%%%%%%%%%%%%%%%%%%%%%%%%%%%%%%%%%%%%%%%%%%%%%%%%%%%%%%%%%%%%%%%%%%%%%%%%%%%%%%%%%%%%%%%%%%%%%%%%%%%%%%%%%%%%%%%%%%%%%%%%%%%%%%%%%%%%%%%%%%%%%%%%%%%%%%%%%%%%%%%%%%%%%%%%%%%%%%%%%%%%%%%

Let us denote the $p$- group $\bigoplus_{i=1}^{n}\mathbb{Z}/p^{\a_i}\mathbb{Z}$ by $G_p$, observe that the abelian group gets a ring 
structure by coordinatewise multiplication, we will call this ring to be $RG_p$. Also let us denote the cluster of an element 
$\underbar{a}=(a_1,a_2,\ldots,a_n)$ by $c(\underbar{a})$, recall that we call support($\underbar{a})=\{i \leq n | a_i \neq 0 \}$. 
The following lemma is provided to make the size of the cluster $c(\underbar{a})$ clear.

\begin{lem}\label{cluster_size}
 Let $\underbar{a} \in G_p$ and let $m$ be the maximum among the $\a_i$ such that $i$ is in the support of $\underbar{a}$ then we have 
 $\card{c(\underbar{a})}=\card{(\mathbb{Z}/p^{m}\mathbb{Z})^*}$.
\end{lem}
\begin{proof}
 Recall that $\underbar{b} \in c(\underbar{a}) \iff \exists x, y  \in \mathbb{Z}$ such that $x\underbar{a}=\underbar{b}$ and $y\underbar{b}=\underbar{a}$. 
 Note that we can take the numbers $x,y$ from the ring $\mathbb{Z}/p^m\mathbb{Z}$, where $m=max\{\a_i \, | \, i \in \mathrm{support}(\underbar{a})\}$. 
 Since $x\underbar{a}=b \implies xy\underbar{b}=\underbar{b} \implies xy=1 \in RG_p$, and if 
 $x\underbar{a}=y\underbar{a}$ then $(x-y)\underbar{a}=\underbar{0} \implies (x-y)=0 \in \mathbb{Z}/p^{\a_i}\mathbb{Z}$ for all $i \in \mathrm{support}(\underbar{a})$. 
 Putting these two together we get the result.
\end{proof}

\begin{cor}
 $\card{c(\underbar{a})}=p^{\a_i -1}(p-1)$, where $\a_i$ is maximum of the set $\{\a_i | i \in \mathrm{support}(\underbar{a})\}$.
\end{cor}

Let us prove a small lemma concerning the group $H=\mathbb{Z}/p^m\mathbb{Z}$, for an element $x \in H$ let us define $r(x)=\card{ \{y \in H \, | \, py=x\} }$.
\begin{lem}\label{r(a)}
 With the notation as above $r(x)=0$ if $x \in (\mathbb{Z}/p^m \mathbb{Z})^*$ else $r(x)=p$. Further if $y \in \Zp{n}^*$ for some $n \geq m$ then $r(yx)=r(x)$.
\end{lem}

In the next lemma we calculate the degree of the clusters in the graph of $G_p$, $C(G_p)$. 
Note that in the cluster graph we only take the so called irreducible edges in consideration and we leave the ``implied'' edges.
\begin{lem}\label{cluster_deg}
 Let $c(\underbar{a})$ where $\ua=(a_1,a_2,a_3,\ldots,a_n)$ be a cluster in the cluster graph $C(G_p)$ then the indegree  of the 
 cluster $\mathrm{indeg}(c(\underbar{a}))=\prod_{i=1}^{n}r(a_i)$.
\end{lem}
\begin{proof}
 The required indegree is given by the number $\card{ \{c(\ub) : p\ub=x \ua \mbox{ such that } x \in (\Zp{\a_i})^* \} }$, 
 where $\a_i$ is the maximum of the numbers $\a_1,\a_2,\a_3, \ldots,\a_n$. Or the number of $\ub$ such that $p(b_1,b_2,b_3,\ldots,b_n)=x(a_1,a_2,a_3,\ldots,a_n)$ or 
 equivalently $pb_i=xa_i$. So the required number is $\prod_{i=1}^{n} r(xa_i)$ where $x \in (\Zp{\a_i})^*$ fixed. Now by the lemma \ref{r(a)} we are done.
\end{proof}

\section{Investigation of Hamiltonian Cycle}

\subsection{Weighted Hamiltonian Cycle}

Given an weighted graph $G=(V,E,w)$ where $w: V \rightarrow \mathbb{N}$ is the weight function we define the graph $\widetilde{G}^w=(\widetilde{V},\widetilde{E})$, 
where $\widetilde(V)=\cup_{v \in V} \{v\}\times \{ 1, \ldots, w(v)\}$ and there is an edge between $(v,i)$ and $(v', i')$ if either $v=v'$ or $(v,v') \in E$. A weighted 
graph as defined above is called a $w$-Hamiltonian graph if there is a path $\pi=v_1v_2v_3\ldots v_n$ such that 
$(v_i,v_{i+1}) \in E$ and $\{v_i | 1 \leq i \leq n\}=V \mbox{ and } \card{\{i| v_i =v \}}\leq w(v)$. Further for such a path $\pi=v_1v_2v_3\ldots v_n$ if $v_1=v_n$ then the path
is called a weighted Hamiltonian cycle or $w-$Hamiltonian cycle. A graph $G$ with weight function $w:V \rightarrow \mathbb{R}$ is called $w-$Hamiltonian if there is a $w-$Hamiltonian
cycle in the graph $G$ with weight $w$.
When the weight is clear from the context we will denote $\widetilde{G}^w$ by $\widetilde{G}$.

In the lemma below we note a very useful fact, which we will henceforth refer to as the generalized cut lemma.
\begin{lem}\label{cut lemma} Let $G=(V,E,w)$ be a weighted graph, let $C \subset V$, if $G$ is $w$-Hamiltonian then $w(C)=\sum_{v \in C} w(v) \geq N$ where $N$ is the 
number of connected components in $G \setminus C$. Further if there exist a $w$-Hamiltonian path then $w(C) \geq N -1$.
 
\end{lem}
\begin{proof} Clearly since there are $N$ connected components in the graph $G\setminus C$ one must visit
$C$ ( counting multiplicity ) at least $N-1$ times to get a $w$-Hamiltonian graph. 
And that is precisely the condition in the statement. 
 
\end{proof}

\begin{lem}
 In a graph $G=(V,E)$ if there is a special vertex $v_0$ such that $(v_0,v) \in E$ for all $v \in V \ \{v_0\}$ then there exist a Hamiltonian 
 cycle in the graph $G$ if and only if there is a Hamiltonian path in the graph $G\setminus v_0$.
\end{lem}
\begin{proof}
It is clear that if $\pi=v_1,v_2,....,v_n$ be a Hamiltonian path in $G \backslash v_0$, then $\pi_1=v_0,v_1,v_2,...,v_n,v_0$
gives a Hamiltonian cycle in $G$.
\end{proof}

In the lemma below we give a necessary and sufficient condition on a weight function $w$ for which any tree $G$ is a $w$-Hamiltonian graph.

\begin{lem}
 A weighted tree $T=(V,E,w)$ is $w$-Hamiltonian if and only if $w(v) \geq deg(v), \forall \, v \in V$.
\end{lem}
\begin{proof} Let the tree $T=(V,E,w)$ be $w$-Hamiltonian and if possible let there exists $v\in V$ such that $deg(v) > w(v)$,
then $\{v\}$ gives a cut, therefore by cut lemma we have $w(v) \geq$ the number of connected components of $G\backslash v$. So we arrive
at a contradiction. Hence $w(v) \geq deg(v), \forall v\in V$. Conversely, let $w(v) \geq deg(v), \forall v \in V$. Now choose 
any $v \in V$ and $T_1,T_2,...,T_k$ be the components of $T\backslash v$. Then $T_i$ also satisfies $w(v_i) \geq deg(v_i),
\forall v_i \in T_i$. If there are Hamiltonian cycles $\alpha_1,\alpha_2,...,\alpha_k$ in $T_1,T_2,....,T_k$ then $v\alpha_1,
v\alpha_2,...,v\alpha_k$ is a Hamiltonian path in $T$ with $w(v) \geq k$ and $w(v_i) \geq deg^{T}(v_i)=deg^{T_i}(v_i)-1$. Hence
it follows.

\end{proof}
  
\begin{lem} Let $(T,w)$ be a weighted tree with weight $w$ (here $T$ may or may not be Hamiltonian). Then $(T,\lambda w)$ is 
Hamiltonian if and only if $\lambda \geq max_{v \in V(T)} \frac{deg(v)}{w(v)}$.
\end{lem}
\begin{proof}$(T,\lambda w)$ is Hamiltonian $\Leftrightarrow \lambda w(v) \geq deg(v), \forall v \in V(T) \Leftrightarrow
\lambda \geq \frac{deg(v)}{w(v)}, \forall v \in V(T) \Leftrightarrow \lambda \geq max_{v \in V(T)} \frac{deg(v)}{w(v)}$.
 
\end{proof}

\begin{thm}
 A weighted graph $G$ with weight $w$ is $w$-Hamiltonian if and only if $\widetilde{G}$ is Hamiltonian.
\end{thm}
\begin{proof} Let the weighted graph $(G,w)$ is $w$-Hamiltonian. Then let $v_0,v_1,v_2,...,v_n$ be a weighted Hamiltonian 
cycle in $G$ such that $\cup \{v_i\}=V$. Let $w_i$'s be the corresponding weights of $v_i$'s. Now consider the cycle 
$(v_0,w_0),(v_1,w_0),...,(v_n,w_0),\\(v_0,w_1),(v_1,w_1),...,(v_n,w_1),(v_0,w_n),...,(v_1,w_n),...,(v_n,w_n),(v_0,w_0)$, which 
gives a Hamiltonian cycle in $\widetilde{G}$. For the other direction, let we have Hamiltonian cycle in  $\widetilde{G}$,
say $(v_0,i_0),(v_1,i_1),....,(v_r,i_r),(v_0,i_0)$. Then take the cycle according to the vertices $v_i$'s we get the weighted 
Hamiltonian cycle.

\end{proof}

%\begin{thm}
 %Let $(G_1,w_1)$ and $(G_2,w_2)$ be two weighted graphs, if $G_1$ is $w_1$ hamiltonian and $w_1(G_1)$ minus the spent weights in the $w_1$ cycle is more than the number $\card{G_2}$ then $G_1 \times G_2$ is $w_1 \times w_2$ hamiltonian.
%\end{thm}
%\begin{proof}
% It is clear that with the remaining weights we can go to any section $(G_1,v)$ and using the fact that this is $w_1$ hamiltonian as $G_1$ we can finish off this section, as the number or remaining weights is more than such sections we can finish off all of the graph.
%\end{proof}

\subsection{The grid }
Let for a prime $p$, $G_p=\bigoplus_{i=1}^{l}\mathbb{Z}_{p^{n_i}}$ for some numbers $n_i \in \mathbb{N}$, we define the level sets $L_k=\{[x] \in G_p | ht(x)= k\}$, 
where $[x]$ denotes the cluster of the 
element $x$ and $ht(x)$ is the number $\log_{p}(o(x))$ or equivalently the 
number of irreducible edges it takes to reach the cluster of $0$ from the cluster of $x$. Now observe that the induced subgraph on $L_0 \cup L_1$ is a star. 
We will call this the associated
star to the group $G_p$. Note that the cluster graph of the group $G_p$ gets a graded structure on the set of vertices through these level sets. Or in other words we have 
$V = \cup_{i=1}^{l} L_i$.

 \subsection{Cluster grid of a graph} 
 Recall a subset $S$ of the set of vertices of a graph $G$ is called a cut if the number of connected components in $G \setminus S$ is larger than the 
 number of elements in the set $S$. In this section we will investigate several cuts in the power graph of a group $G$. Consider the class of groups
  $G=\mathbb{Z}_p^{m} \times \mathbb{Z}_q^{n}$ let $G(m,n,w)$ be the weighted cluster grid of the group. 
 \begin{lem}\label{simple_cuts}
  In the graph $G(m,n,w)$ as in above we have:
  \begin{enumerate}
   \item $S=\{ (x,y) | x=0, y \leq n\}$ is a cut if $n >m(p-1)+1$.
   \item $S=\{(x,y)| x \leq m , y=0\}$ is a cut if $m > n(q-1)+1$.
   \item $S=\{(x,y)| x=0 \, \mathrm{ or } \, y=0 \}$ is a cut if $mn > m(p-1)+n(q-1)+1$.
  \end{enumerate}

 \end{lem}
\begin{thm}\label{main_necessary}
 The power graph of a the group $G=(\Z_p)^m\times(\Z_p)^n$ is not Hamiltonian if $n > m(q-1) +1$ or $m > n(q-1)+1$ or $mn > m(q-1)+n(q-1)$.
\end{thm}
\begin{proof}
 Follows trivially from the lemma \ref{simple_cuts}.
\end{proof}

We can generalize the above lemma for the general grid $Grid_{m_1,m_2,m_3,\ldots,m_r}^{u_1,u_2,u_3,\ldots u_r}$. 
Let us take a subset $I \subset [r]=\{1,2,\ldots,r\}$ We define the cut $\mathfrak{I}$ associated to the subset $I$ is 
the set of vertices $\cup_{i \in I} C_i$ where $C_i=\{(a_1,a_2,a_3,\ldots,a_r): a_i =0\}$. In this following lemma we will calculate the total 
weight of the cut and calculate the number of components in it's complement. For a subset $A \subset Grid_{m_1,m_2,m_3,\ldots,m_r}^{u_1,u_2,u_3,\ldots u_r}$ 
let us call the dimension of the span of $A$ as the dimension of $A$. So using this notation dimension of $C_i$ is $r-1$ or the co-dimension is 1. 

\begin{lem} Let $A$ be a subset of the grid $Grid_{m_1,m_2,m_3,\ldots,m_r}^{u_1,u_2,u_3,\ldots u_r}$, if the co-dimension of the subset $A$ is more than 1 
then it does not disconnect the grid. 
 
\end{lem}
\begin{proof}
 We will show that $C_{ij}=\{ (a_1,a_2,a_3,\ldots , a_r) | a_i = a_j =0 \}$ for $i \neq j$ and unions of these does not disconnect the grid. 
 And a similar argument will prove that for even a higher co-dimension subsets will not disconnect these. Let us show that the points 
 $(\a_1,\a_2,\a_3, \ldots , \a_r)$ with $\a_i , a_j $ not both zero is connected by a path to $(\b_1,\b_2,\b_3,\ldots, \b_r)$ with $\b_i,\b_j$ not both zero. 
 First let us show the case that $\a_i \neq 0 $ and $\b_j \neq 0 $, we have the path 
 $(\a_1,\a_2,\a_3,\ldots \a_r) \rightarrow (\a_1,\a_2,\ldots, \a_i \ldots , 0 , \ldots \a_r)$ 
 ( where we have 0 in the $j$th position, $\rightarrow (\a_1,\a_2,\a_3, \ldots \a_i , \ldots , \b_j , \ldots \a_r)$. Now we use the fact that $\b_j \neq 0 $ 
 to continue the path with $\rightarrow (\a_1,\a_2,\a_3, \ldots , 0, \ldots , \b_j , \ldots , \a_r)$ where the 0 appears in the $i$th position. Continue as 
 $\rightarrow (\a_1,\a_2, \a_3 , \ldots \b_i , \ldots \b_j , \ldots \a_r)$. Now we can change the rest of the points similarly. For the case $\a_i = \b_i =0 $ 
 we will start with $(\a_1,\a_2, \a_3, \ldots , \a_i , \ldots \a_j , \ldots \a_r) \rightarrow (\a_1, \a_2, \a_3, \ldots , 1, \ldots , \a_j , \ldots \a_r)$ now 
 the final point is connected to $( \b_1, \b_2, \b_3 , \ldots , \b_i , \ldots , \b_j , \ldots , \b_r) $ as the first point has 1 in the $i$th place and the second
 one has a nonzero entry in the $j$th place.
\end{proof}
As a corollary of the above lemma one can say that the only subsets that disconnects the grid are unions of $C_i$ for $i \in I$.

\begin{lem} Let $Grid_{\underline{m}}^{\underline{u}}$, $I$ and $\mathfrak{I}$ be as above then,
 \begin{enumerate}
  \item $w(I)=w(\mathfrak{I})=\displaystyle\sum_{\ua \in \mathfrak{I} } w(\ua)=\displaystyle\prod_{j \notin C} u_j \displaystyle\sum_{A \subsetneq C , \, j \in A} m_ju_j+1$
  \item number of component in $Grid_{\underline{m}}^{\underline{u}} \setminus \mathfrak{C}=\displaystyle\prod_{j \in C} m_j$.
 \end{enumerate}

\end{lem}
\begin{proof}
 \begin{enumerate}
  \item \[w(\ua)=\sum_{\ua \in \mathfrak{C}} w(\ua)=\sum_{\ua \in \mathfrak{C} } w(a_1)w(a_2)\ldots w(a_r)\] Since for $a_i \neq 0 $ $w(a_j)=u_j$.
  \[ = \sum_{\ua \in \mathfrak{C} } \prod_{a_j \neq 0} u_j\] Now we take the product $\prod_{j \notin C} u_j$ out. Now since for $A \subsetneq C$ for $(a_1,a_2,a_3,\ldots,a_r)$ such that $a_i =0 \mbox{ for } i \in C \setminus A$, we have $w(a_1,a_2,a_3,\ldots,a_r)= \prod_{j \notin A} u_j$, and since there are $\prod_{j \not in A} m_j$ such elements we have the next equality.
  \[= \prod_{j \notin C} u_j \sum_{A \subsetneq C} \prod_{j \notin A} m_ju_j\]
\item For this part note that if $\ua, \ub \in Grid_{\underline{m}}^{\underline{u}} \setminus \mathfrak{C}$ then $\ua$, $\ub$ are not in the same connected component if $\{i : a_i \neq b_i\} \cap C \neq \emptyset$. So the number of connected components are equal to the number of $a_i$ such that $i \in C$, or $\prod_{j \in C} m_j$.
 \end{enumerate}

\end{proof}
As a consequence of the (subsets ) described in the above lemma we have the following necessary conditions:
\begin{lem} If for any subset $C \subset [r]$, $w(C) > \prod_{j \in C} m_j$ then the grid $Grid_{\underline{m}}^{\underline{u}}$ is not Hamiltonian.
 
\end{lem}
\begin{proof}
 It's an immediate consequence of the lemma \ref{cut lemma}.
\end{proof}

 %\subsection{Towards sufficient conditions}

%%%%%%%%%%%%%%%%%%%%%%%%%%%%%%%%%%%%%%%%%%%%%%%%%%%%%%%%%%%%%%%%%%%%%%%%%%%%%%%%%%%%%%%%%%%%%%%%%%%%%%%%%%%%%%%%%%%%%%%%%%%%%%%%%%%%%%%%%%%%%%%%%%%%%%%%%%%%%%%%%%%%%%%%%%%%%%%%%%%%%%%%%%%%%%%%%%%%%%%%%%%%%%%%%%%%%%%%%%%%%%%%%%%%%%%%%%%%%%%%%%%%%%%%%%%%%%%%%%%%%%%%%%%%%%%%%%%%%%%%%%%%%%%%%%%%%%%%%%%%%%%%%%%%%%%%%

%%%%%%%%%%%%%%%%%%%%%%%%%%%%%%%%%%%%%%%%%%%%%%%%%%%%%%%%%%%%%%%%%%%%%%%%%%%%%%%%%%%%%%%%%%%%%%%%%%%%%%%%%%%%%%%%%%%%%%%%%%%%%%%%%%%%%%%%%%%%%%%%%%%%%%%%%%%%%%%%%%%%%%%%%%%%%%%%%%%%%%%%%%%%%%%%%%%%%%%%%%%%%%%%%%%%%%%%%%%%%%%%%%%%%%%%%%%%%%%%%%%%%%%%%%%%%%%%%%%%%%%%%%%%%%%%%%%%%%%%%%%%%%%%%%%%%%%%%%%%%%%%%%%%%%%%%

%%%%%%%%%%%%%%%%%%%%%%%%%%%%%%%%%%%%%%%%%%%%%%%%%%%%%%%%%%%%%%%%%%%%%%%%%%%%%%%%%%%%%%%%%%%%%%%%%%%%%%%%%%%%%%%%%%%%%%%%%%%%%%%%%%%%%%%%%%%%%%%%%%%%%%%%%%%%%%%%%%%%%%%%%%%%%%%%%%%%%%%%%%%%%%%%%%%%%%%%%%%%%%%%%%%%%%%%%%%%%%%%%%%%%%%%%%%%%%%%%%%%%%%%%%%%%%%%%%%%%%%%%%%%%%%%%%%%%%%%%%%%%%%%%%%%%%%%%%%%%%%%%%%%%%%%%

%%%%%%%%%%%%%%%%%%%%%%%%%%%%%%%%%%%%%%%%%%%%%%%%%%%%%%%%%%%%%%%%%%%%%%%%%%%%%%%%%%%%%%%%%%%%%%%%%%%%%%%%%%%%%%%%%%%%%%%%%%%%%%%%%%%%%%%%%%%%%%%%%%%%%%%%%%%%%%%%%%%%%%%%%%%%%%%%%%%%%%%%%%%%%%%%%%%%%%%%%%%%%%%%%%%%%%%%%%%%%%%%%%%%%%%%%%%%%%%%%%%%%%%%%%%%%%%%%%%%%%%%%%%%%%%%%%%%%%%%%%%%%%%%%%%%%%%%%%%%%%%%%%%%%%%%%

\section{Grid Algorithms}

Let us take the weighted grid graph $Grid_{m,n}^{u,v}$ where $m,n,u,v \in \mathbb{N}$ is defined as the graph with $V=\{0,1,2,\ldots,m\}\times\{0,1,2,\ldots,n\}$ 
and directed edges are $((i,j),(l,k))$ if $i=0 \mbox{ and } j=0$ or $ i =l \mbox{ and } j=0 $ or $i =0 \mbox{ and } j =k$. Note that this is just the product 
graph of the stars $S_m =[m] \cup \{0\}$ and edges from $ i \in [m] $ to $0$, and $S_n$. And the weight is given by 
$w(i,j)=uv \mbox{ if } j \neq 0 \mbox { and } j \neq 0 $ and $w(0,j)=v \mbox{ and } w(i,0)=u$. In the light of the above section 
it is noted that the Hamiltonicity of such grids are one of the basic questions that we have to discuss. In this section we will discuss such graphs in detail.
Note that we can generalize the grid to a general grid as follows $Grid_{m_1,m_2,\ldots,m_t}^{u_1,u_2,\ldots,u_t}=Grid_{\underbar{m}}^{\underbar{u}}$ 
is a graph of the product graph of stars $S_{m_i}$ where the weight of a vertex $w(x_1,x_2,x_3,\ldots,x_t)=\prod_{i=1}^{t}u_i(x_i)$ with the understanding 
that $u_i : \{ 0,1,2,\ldots,m_i\}\to \mathbb{N}$ is a function with $u_i(0)=1$. If we take the functions $u_i(x)=u_ix \mbox{ for } x \neq 0 \mbox{ and } 0 \mbox{ otherwise }$ 
then we call that a simple grid.
\begin{thm}
 Let $G\mbox{rid}_{(m,n)}^{(v,u)}=(V,E,w)$ is $w$-Hamiltonian then $nu+mv \geq mn -1$ and $mv \geq n-1$.
\end{thm}
\begin{proof}
 Let us take $C=\{(i,j)\in V : i=0 \mbox{ or } j =0 \}$. Note that $w(C)$ as defined in the lemma \ref{cut lemma} is $ \sum_{v \in C} w(v)= nu+mv+1$. 
 And the number of connected components in the complement of $C$ is clearly is the set of singleton points $(i,j)$ such that $i \neq 0 \mbox{ and } j \neq 0$, so the number
 of components is $mn$. Now the result follows from the lemma \ref{cut lemma}.
\end{proof}

\begin{thm}
 The grid as defined above with all the notations in place with $n \geq m$ is Hamiltonian if $nu +mv \geq mn-1$ and $mv \geq n-1$.
\end{thm}
\begin{proof}
 To prove the above theorem we will make use of the following lemmas.
\end{proof}

The complete bipartite graph $K_{m,n}=(V,E)$ where $V=[m]\times\{0\} \cup \{0\}\times [n]$ with weights defined as $w(i,0)=v$ and $w(0,j)=u$ will be called the associated complete bipartite graph of the grid $Grid_{\underbar{m}}^{\underbar{u}}$. 

For a weighted complete bipartite graph $K_{m,n}$ with vertex set $V=\{x_1,x_2,x_3 ,\ldots x_n\} \sqcup \{y_1 , y_2,y_3,\ldots y_m\}$ with vertex weight $w$ we define an acceptable path in the following way. For a path $\pi = v_1v_2v_3 \ldots v_k$ we define $n^{\pi}(v)=\card{ \{ l : v_l = v\} }$ or in other words the number of times the vertex $v$ appears in the path $\pi$.
\begin{defn}\label{acceptable} A $w$-Hamiltonian path $\pi=v_1,v_2,...,v_k$ in $K_{m,n}$ is acceptable
if there is a partition of edges not in the path $ E\backslash \{(v_i,v_{i+1}): 1\leq i\leq k-1\}=A\sqcup B$, where $A=\cup A_i,B=\cup B_j, A_i=\{l:(x_l,y_i)\in A\},
B_j=\{l: (x_j,y_l)\in B\}$, such that the following two statements are true.
\begin{enumerate}

\item $w(x_i)\geq n^{\pi}(x_i)+\lvert B_i \rvert -1$ if $v_1=x_i$,
else, $w(x_i)\geq n^{\pi}(x_i)+\lvert B_i \rvert$.

\item $w(y_j)\geq n^{\pi}(y_j)+\lvert A_j \rvert -1$ if $v_1=y_j$
else, $w(y_j)\geq n^{\pi}(y_j)+\lvert A_j \rvert$.
\end{enumerate}

\end{defn}
\begin{prop}\label{grid_bipartite_equivalence} $G\mbox{rid}_{(m,n)}^{(v,u)}$ is $w$-Hamiltonian if and only if there is an acceptable path in the associated complete bipartite graph.
\end{prop}

\begin{proof}
 Let us prove that if there is a $w$-Hamiltonian cycle in the grid then there is an acceptable path in the associated 
 complete bipartite graph. Without loss of generality we assume that the $w$-Hamiltonian cycle $\pi =v_0v_1\ldots v_r v_{r+1}$ 
 starts at the point $(0,0)$, i.e $v_0 = v_{r+1} = (0,0)$.
 For the rest of the vertices $v_i=(x_{j_i},y_{j_i})$, we call such a vertex interior vertex if $x_{j_i}y_{j_i} \neq 0$. We define a 
 path $\tilde{\pi}$ on the associated complete bipartite graph by ignoring the interior vertices of the path $\pi$ and realizing the points 
 $(x_i,0)$ as $(i,0)$ and $(0,y_i)$ as $(0,i)$. Clearly this is a $w$-Hamiltonian path. For the acceptable part define 
 $A_{i}=\{l : (x_i,0)(x_i,y_l)(x_i,0) \mbox{ is an expression in } \pi \}$, similarly we define 
 $B_{j}=\{ l : (0,y_j)(x_l,y_j)(0,y_j) \mbox{ is an expression in } \pi \}$. It is easy to check that this is an acceptable path. 
 For the reverse direction we define the $w$-Hamiltonian path in the grid in the following way, we take an 
 acceptable path $\pi=v_1 v_2 v_3 \ldots v_r$ on $K_{m,n}$ and at the 
 first point when $v_i=(x_j,0) $ we insert the expression $\prod_{l \in B_j}(x_j,y_{l_i})(x_j,0)$ and 
 whenever first we have $v_k=(0,y_t)$ we insert the expression $\prod_{s \in A_k} (x_{l_a},y_y)(0,y_t)$. It is 
 easy to prove that this is a $w$-Hamiltonian cycle in the grid.
\end{proof}

Let us consider the path as described below in the complete bipartite graph $K_{m,n}$. 
Let us assume that $V_1=\{x_1,x_2,x_3,\ldots,x_m\}, V_2=\{y_1,y_2,y_3,\ldots, y_n\}$ and also without a loss of generality let us assume 
that $m \geq n$. Let us also consider the weight $w$ which is defined as $w(x_i)=v$ and $w(y_j)=u$. Also assume that $m=sn+r$, where 
$0 \leq r < n$. Let us look at the following path $\pi=v_1v_2v_3\dots v_{2m-1}$:
\[x_1y_1x_2y_2,\cdots,x_ny_n\]
\[x_{n+1}y_1x_{n+2}y_2, \cdots, x_{2n}y_n\]
\[.\]
\[.\]
\[.\]
\[x_{n(s-1)+1}y_1x_{n(s-1)+2} \cdots x_{ns}y_n\]
\[x_{ns+1}y_1x_{ns+2}y_2 \cdots x_{ns+r}\]
So we can also say that for $j <m$:
\[ v_j =
  \begin{cases} 
      \hfill x_l   \hfill & \text{ if $j$ is odd and } l=(j+1)/2\\
      \hfill y_k \hfill & \text{ if $j$ is even } k= j  \text{ mod } (n )\\
  \end{cases}
\]
Let us define for a vertex $a \in V_1 \cup V_2$ the number of times the vertex occurs in the path $\pi$, $d(a)=\card{v_j : v_j =a }$, in the following lemma we calculate these numbers precisely for the path.

\begin{lem} For the path $\pi$ as above we have the following equalities;
\begin{enumerate}
 \item \[
 d(y_i) =
  \begin{cases} 
      \hfill s    \hfill & \text{ if $i \geq r$} \\
      \hfill s+1 \hfill & \text{ else } \\
  \end{cases}
\]
\item $d(x_i) =1 $ for all $i$.
\end{enumerate}

\end{lem}
\begin{proof}
 It follows clearly from the definition of the path.
\end{proof}

\begin{lem} For the path $\pi$ above we have $d(y_i) \leq u$.
\end{lem}
\begin{proof}
 \begin{enumerate}
  \item \underline{Case $r=0$}. In this case $d(y_i)=m/n$ and we know that $nu \geq m-1$ so $u \geq (m-1)/n=m/n -1/n$ and both $u$ and $m/n$ are integers we have the result of the lemma.
  \item \underline{ Case $i \geq r$}. We have $d(y_i) = s $ so $md(y_i)=ns = m-r \leq nu+1 -r \leq nu $ hence $d(y_i) \leq u$.
  \item \underline{Case $ 1< i < r$}. $d(y_i)=s+1$ or $n d(y_i)=ns +n =m+n -r \leq nu+1 +n -r $. Hence $d(y_i) \leq u + 1+n -r / n \leq u$ as $1+n -r < n$.  
 \end{enumerate}

\end{proof}

Now let us consider the weights we have after the path $\pi$ for the first $x_1$ and the last $x_m$ visits there will be no weights lost as we are starting from this points or still to move from this points. So the weights used will be for each of the $x_i$ where $i $ is other than $1,m$ we have used up only one weight. And for the $y_i$ we have used as many weights as we have visited these points, that is $d(y_i)$. Let us denote the available weights by $\theta(a)$ for a vertex $a \in V_1 \cup V_2$. And note that we have the following weights left:
\[
 \theta(x_i) =
  \begin{cases} 
      \hfill v    \hfill & \text{ if $i= 1, m$ } \\
      \hfill v-1 \hfill & \text{ else} \\
  \end{cases}
\]
and
\[
 \theta(y_i) =
  \begin{cases} 
      \hfill u-s    \hfill & \text{ if $i \geq r$} \\
      \hfill u-s-1 \hfill & \text{ else} \\
  \end{cases}
\]

Consider the subset $A_{\pi}=\{(i,j) : (x_i,y_j) \mbox{ does not occur in } \pi \}$ of the grid $[m]\times [n]$, with weights $u_i = \theta(x_i)$ and $v_j= \theta(y_j)$. In the following proposition we will prove that this subset is colorable a concept that we are defining in the next subsection see \ref{colorable_def}.
\begin{prop}
 The subset $A_{\pi}$ as defined above is colorable.
\end{prop}
\begin{proof}
 As a result of \ref{colorable_def} we have to prove that for any subset $B \subset A_{\pi}$ we have the inequality:
 \[\displaystyle{ \sum_{i \in \pi_1(B)} u_i + \sum_{j \in \pi_2(B)} v_j \geq \card{B}} \]
 \[\mbox{ or } \displaystyle{ \sum_{i \in \pi_1(B)} \theta(x_i) + \sum_{j \in \pi_2(B)} \theta(y_j) \geq \card{B}}\]
 To deal with the case when $1,m \in \pi_1(B)$  we see that the maximum cardinality of such a $B$ is the size of the grid minus the number of edges appearing in the path $\pi$. Clearly there are $2m-1$ edges in the path $\pi$ so the maximum size of $B$ is $mn -2m +1$ in such case the left hand side becomes: $(v-1)(m-2)+2v + (u-s)(n-r) + (u-s-1)r= mv+nu -2m +2$ and since we have $mv + nu \geq mn -1$ we have the desired colorability in this case. The general case is an easy reiteration of this argument.

\end{proof}

Let us summarize the above in this following theorem.
\begin{thm}\label{acceptability_main}
 The associated complete bipartite to the grid $Grid_{m,n}^{v,u}$ has an acceptable path if $mv \geq n-1 , nu \geq m-1$ and $mv +nu \geq mn -1$.
\end{thm}

\subsection{Color game} In this section we will build up the subsets $A,B$ for the above path using a game we will call color game. 
Let us consider a subset $A \subset [m]\times [n]$ where $[m]=\{1,2,3\ldots,m\}$, and there are natural numbers $u_1,u_2,u_3,\ldots ,u_n$ and $v_1,v_2,v_3, \ldots, v_m$. 
The points $(i,j) \in A$ could be colored by either color red if $u_i \geq 1$ and if it is colored red then $u_i $ is reset by $u_i-1$ and the point 
could be colored blue if $v_j \geq 1$ and the number $v_j$ is reset by $v_j-1$. The goal is to color all the points on $A$ with some colors. If there is a 
legal coloring of the subset with the given weights, so that no weight is allowed to go less than zero, then the subset is called colorable with the given weights.
\begin{lem}\label{colorable_def}
 A subset $A \subset [m] \times [n]$ is colorable if and only if for every subset $B \subset A$ we have $\displaystyle{ \sum_{i \in \pi_1(B)} u_i + \sum_{j \in \pi_2(B)} v_j \geq \card{B}}$.
\end{lem}
\begin{proof}
 The necessary condition is obvious. For the sufficient condition let us give an algorithm:
 \begin{enumerate}
  \item Pick a column $i$ such that $\pi^{-1}(i)$ has the maximum cardinality.
  \item If the color available on the $X$ axis is not zero color the point with red from $(i,j_1), (i,j_2), (i,j_3) \ldots, (i , j_r)$ where weights $v_{j_1} \leq v_{j_2} \leq \ldots v_{j_r}$.
  \item Color the rest of the points on $\pi^{-1}(i)$ blue. This can be done since the condition for this subset guarantees that.
  \item Now observe that $A'= A \setminus \pi^{-1}(i)$ is a subset of a smaller grid and by the lemma \ref{color_1}, $A'$ could be colored by the remaining colors.
 \end{enumerate}

\end{proof}
\begin{lem}\label{color_1} Let $A'$ be as above then for every subset $B \subset A'$ we have \[\displaystyle{ \sum_{i \in \pi_1(B)} x_i + \sum_{j \in \pi_2(B)} y_j \geq \card{B}}.\]
 
\end{lem}
\begin{proof}
 If we realize the subset $B \subset A'$ as a subset of $A$ we get the desired inequality.
\end{proof}

Continuing with the notations from the definition \ref{acceptable} we will denote the associated complete bipartite weighted graph to the grid $Grid_{\underbar{m}}^{\underbar{u}}$ as $K_{m,n}$ with weight $w$. Let $\pi=v_1v_2v_3\ldots v_l$ be a $w$-Hamiltonian path in $K_{m,n}$. Let $E'$ be all the edges encountered in this path, let $A_{\pi}=\{(i,j) | (i,j) \notin E'\}$, and let us take $u_i = w(x_i) - \card{\{l : v_l =x_i\}}$ if $x_i \neq v_1$ and $w(x_i)-\card{ \{ l : v_l =x_i \} } +1$ else, similarly we define $v_j$. In the lemma below we characterize the acceptable paths.
\begin{lem} The subset $A_{\pi}$ is colorable if and only if the path $\pi$ is acceptable.
\end{lem}
\begin{proof}
 If $A_{\pi}$ is colorable then define $A_l=\{y_i | (l,i) \mbox{ is colored red } \}$ and similarly we define $B_l=\{x_i | (i,l) \mbox{ is colored blue } \}$. Note that since $A_{\pi}$ is colorable we have
 \begin{enumerate} 
 
\item $w(x_l) - n^{\pi}(x_l) \geq \card{A_l}$ if $x_l \neq v_1$ else $w(x_l)-n^{\pi}(x_l) +1 \geq \card{A_l}$.
 \item $w(y_l) - n^{\pi}(y_l) \geq \card{B_l}$ if $y_l \neq v_1$ else $w(y_l) - n^{\pi}(y_l) +1 \geq \card{B_l}$.
 \end{enumerate}
 This finishes the proof of the acceptability and for the reverse part we color the vertices in $A$ red and those in $B$ blue, the condition for acceptability permits us to have a legal coloring with these colors.
\end{proof}

\begin{cor}\label{main_corollary}
 The grid $G\mbox{rid}_{(m,n)}^{(v,u)}=(V,E,w)$ is $w-$Hamiltonian if and only if $nu+mv \geq mn -1$ and $nu \geq m-1, mu \geq n-1$.
\end{cor}
\begin{proof}
 Necessary condition was shown in \ref{main_necessary}, and the sufficient condition follows from \ref{acceptability_main} and \ref{grid_bipartite_equivalence}.
\end{proof}

\begin{thm}\label{main}
 The powergraph of the Abelian groups $(\mathbb{Z}_p)^n \times (\mathbb{Z}_q)^m$, where $p,q$ are distinct primes, is Hamiltonian if and only if $m(q-1) \geq n-1$ and $n(p-1) \geq m-1$ and $m(q-1) + n(p-1) \geq mn -1$. 
\end{thm}
\begin{proof}
 Follows from the corollary \ref{main_corollary}
\end{proof}

\bibliographystyle{abbrv}
\bibliography{power}

\end{document}